\documentclass[12pt]{amsart}
\usepackage{amssymb}
\usepackage[shortlabels]{enumitem}
\usepackage[all]{xy}
\usepackage{xcolor}
\usepackage{graphicx}
\usepackage[margin=1in]{geometry} 
\usepackage{hyperref}
\hypersetup{bookmarksdepth=2}
\hypersetup{colorlinks=true}
\hypersetup{linkcolor=blue}
\hypersetup{citecolor=blue}
\hypersetup{urlcolor=blue}

\numberwithin{equation}{section}
\newtheorem{theorem}[equation]{Theorem}

\newtheorem{proposition}[equation]{Proposition}
\newtheorem{lemma}[equation]{Lemma}

\theoremstyle{definition}
\newtheorem{remark}[equation]{Remark}
\newtheorem{definition}[equation]{Definition}

\newcommand{\arxiv}[1]{\href{http://arxiv.org/abs/#1}{{\tiny\tt arXiv:#1}}}
\newcommand{\DOI}[1]{\href{http://doi.org/#1}{\color{purple}{\tiny\tt DOI:#1}}}
\newcommand{\defn}[1]{\emph{#1}}

\newcommand{\bC}{\mathbf{C}}
\newcommand{\bF}{\mathbf{F}}
\newcommand{\bZ}{\mathbf{Z}}
\newcommand{\FF}{\mathbb{F}}
\newcommand{\JJ}{\mathbb{J}}
\newcommand{\MM}{\mathbb{M}}

\title{A characterization of Jacobi sums}

\author{Andrew Snowden}
\thanks{The author was supported by NSF grant DMS-2301871.}
\address{Department of Mathematics, University of Michigan, Ann Arbor, MI, USA}
\email{\href{mailto:asnowden@umich.edu}{asnowden@umich.edu}}
\urladdr{\url{http://www-personal.umich.edu/~asnowden/}}

\date{November 22, 2024}

\begin{document}

\begin{abstract}
Let $\MM$ be the group of multiplicative characters of a finite field $\FF$, and let $\JJ(\alpha, \beta)$ be the Jacobi sum, for $\alpha, \beta \in \MM$. We observe that the function $\JJ \colon \MM \times \MM \to \bC$ satisfies three elementary properties. We show that these properties (very nearly) characterize Jacobi sums: if $M$ is an arbitrary non-trivial finite abelian group and $J \colon M \times M \to \bC$ is a function satisfying these properties then $M$ is naturally the group of multiplicative characters of a finite field and $J$ is the Jacobi sum.
\end{abstract}

\maketitle
\tableofcontents

\section{Introduction}

\subsection{Jacobi sums}

Let $\FF$ be a finite field with $q$ elements and let $\MM$ be the group of complex valued characters of the multiplicative group $\FF^*$. For $\alpha, \beta \in \MM$, we define the \defn{Jacobi sum} by
\begin{displaymath}
\JJ(\alpha, \beta) = \frac{1}{q-1} \sum_{\substack{x+y=1, \\ x,y \in \FF^*}} \alpha(x) \beta(y).
\end{displaymath}
We note that many authors modify the definition when one or both characters are trivial, and also do not divide by $q-1$. Jacobi sums satisfy the following identities (see \S \ref{s:ident}):
\begin{itemize}
\item[(A)] We have $\JJ(\alpha, \beta)=\JJ(\beta, \alpha)$ for all $\alpha, \beta \in \MM$.
\item[(B)] Letting $\JJ^*(\alpha, \beta)=-\delta(\alpha)-\delta(\beta)+\JJ(\alpha,\beta)$, we have
\begin{displaymath}
\JJ^*(\alpha, \beta) \JJ^*(\alpha \beta, \gamma) = \JJ^*(\alpha, \beta \gamma) \JJ^*(\beta, \gamma).
\end{displaymath}
for all $\alpha, \beta, \gamma \in \MM$. Here $\delta(\alpha)$ is~1 if $\alpha$ is the trivial character, and~0 otherwise.
\item[(C)] For all $\alpha_1, \alpha_2, \alpha_3, \alpha_4 \in \MM$, we have
\begin{displaymath}
\sum_{\beta \in \MM} \JJ(\alpha_1 \beta, \alpha_2 \beta^{-1}) \JJ(\alpha_3 \beta, \alpha_4 \beta^{-1}) = \JJ(\alpha_1 \alpha_4, \alpha_2 \alpha_3).
\end{displaymath}
\end{itemize}
Our goal is to show that these properties (very nearly) characterize Jacobi sums.

\subsection{The main theorem}

To formulate our result, it will be convenient to introduce the following piece of  terminology:

\begin{definition}
A \defn{Jacobi function} on a finite abelian group $M$ is a function
\begin{displaymath}
J \colon M \times M \to \bC
\end{displaymath}
satisfying the analogs of (A), (B), and (C).
\end{definition}

Of course, the main example of a Jacobi function is a Jacobi sum $\JJ$ associated to a finite field; we refer to this as the ``classical case.'' In what follows, we let $\hat{M}$ denote the Pontryagin dual of a finite abelian group $M$, and we write $\alpha(x)$ for the pairing between $\alpha \in M$ and $x \in \hat{M}$. We are now ready to state our main theorem:

\begin{theorem}
Let $M$ be a finite abelian group of order $m \ge 2$, and let $J$ be a Jacobi function on $M$. Then there is a unique field structure on the set $F=\hat{M} \sqcup \{0\}$, where~0 is the additive identity and the multiplication extends the group law on $\hat{M}$, such that
\begin{displaymath}
J(\alpha, \beta) = \frac{1}{m} \sum_{\substack{x+y=1,\\ x,y \in F \setminus \{0\}}} \alpha(x) \beta(y).
\end{displaymath}
\end{theorem}

We say a little about the proof. Let us first return to the classical case, where $J=\JJ$ is the Jacobi sum for a finite field $\FF$. Let $X \subset (\FF^*)^2$ be the set of pairs $(x,y)$ where $x+y=1$, and identify $\FF^*$ with $\hat{\MM}$. The function $\JJ$ is exactly the Fourier transform of the indicator function of the set $X$. Thus, given $\JJ$, we can ``discover'' $X$ by taking an inverse Fourier transform. With $X$ in hand, it is easy to reconstruct the addition law on $\FF$. This is how the proof works in general: given an arbitrary Jacobi function $J$, we show that its inverse Fourier transform is the indicator function of a set, and use this set to construct the addition law on $F$.

\subsection{Remarks} \label{ss:rmk}

We make a number of remarks concerning the theorem.

(a) Suppose $M$ is the trivial group. A Jacobi function $J$ is then determined by the single value $a=J(1,1)$. Conditions (A) and (B) are trivial, while (C) amounts to the equation $a^2=a$. We thus find two solutions to the equations, namely, $a=0$ and $a=1$. The $a=0$ case is the Jacobi sum for the finite field $\bF_2$. The $a=1$ case does not come from a finite field structure, but it can be viewed as the ``Jacobi sum'' for the boolean semi-ring.

(b) Given a Jacobi function $J$ on $M$, one can obtain new Jacobi functions on $M$ by pre-composing $J$ with a group automorphism of $M$, or post-composing $J$ with a Galois automorphism of $\bC$. We now explain how these operations work in the classical situation. Thus suppose $M=\MM$ and $J=\JJ$ for some finite field $\FF$ of cardinality $q$.

Let $r$ be an integer coprime to $q-1$ so that $\alpha \mapsto \alpha^r$ is an automorphism of $\MM$. We thus have a Jacobi function $\JJ'$ on $\MM$ by $\JJ'(\alpha, \beta)=\JJ(\alpha^r, \beta^r)$. The map $x \mapsto x^r$ is a bijection $\FF \to \FF$ that is compatible with multiplication, and so we can define a new field structure on $\FF$ by conjugating addition by this map. The Jacobi function $\JJ'$ is the classical Jacobi sum with respect to this new field structure.

Next, suppose that $\sigma$ is a field automorphism of $\bC$. We thus obtain a new Jacobi function $\JJ''$ on $\MM$ by $\JJ''=\sigma \circ \JJ$. Now, there is then an integer $r$ coprime to $q-1$ such that $\sigma(\zeta)=\zeta^r$ for all $q-1$ roots of unity $\zeta \in \bC$. One easily sees that $\JJ''$ is exactly the Jacobi function $\JJ'$ defined above. Thus it too is the classical Jacobi sum for the new field structure on $\FF$.

(c) There are a number of possible generalizations that one can consider. (i) One can try to classify Jacobi functions valued in an arbitrary commutative ring. (ii) The beta function is often considered to be the real number analog of Jacobi sums. Is there a real number analog of our theorem? (iii) One can consider Jacobi sums where the finite field $\FF$ is replaced with a finite ring such as $\bZ/p^n \bZ$ (see, e.g., \cite[\S 5]{Nica}), or even a finite semi-ring. One can then attempt to characterize these. We will elaborate on this in \cite{jaccat}.

(d) In recent work with Harman \cite{repst}, we introduced a certain notion of measure on oligomorphic groups, and showed that they can be used to construct interesting new tensor categories. Given a finite abelian group $M$, there is an oligomorphic group $G=G(M)$ that is defined as the symmetry group of an $M$-weighted Cantor set. It turns out that measures for $G(M)$ correspond to Jacobi functions on $M$. This is how we were led to this particular set of conditions. We will discuss the details of this in forthcoming work \cite{jaccat}.

\section{Reformulation of the second condition}

Fix a finite abelian group $M$ and a function $J \colon M \times M \to \bC$. Condition (B) is formulated in terms of the auxiliary function $J^*$. We now give a direct formulation in terms of $J$ which, while more complicated, will be at times be more convenient.

\begin{proposition} \label{prop:reform}
Condition (B) is equivalent to the identity
\begin{displaymath}
\scalebox{.95}{$\displaystyle J(\alpha, \beta) J(\alpha \beta, \gamma)-J(\alpha, \alpha^{-1})  \delta(\alpha \beta)+\delta(\alpha) \delta(\beta)
= J(\alpha, \beta \gamma) J(\beta, \gamma)-J(\gamma, \gamma^{-1})  \delta(\beta \gamma)+\delta(\beta) \delta(\gamma).$}
\end{displaymath}
\end{proposition}

\begin{proof}
Let $A$ and $B$ be the two sides of the above equation, and put
\begin{displaymath}
A^* = J^*(\alpha, \beta) J^*(\alpha \beta, \gamma), \qquad
B^* = J^*(\alpha, \beta \gamma) J^*(\beta, \gamma).
\end{displaymath}
By definition of $J^*$, we have
\begin{displaymath}
A^* = (J(\alpha, \beta) - \delta(\alpha) - \delta(\beta))(J(\alpha \beta, \gamma) - \delta(\alpha \beta) - \delta(\gamma))
\end{displaymath}
We now expand the product to obtain
\begin{align*}
A^* =& J(\alpha, \beta) J(\alpha \beta, \gamma) - J(\beta, \gamma) \delta(\alpha) - J(\alpha, \gamma) \delta(\beta) -J(\alpha, \beta) \delta(\gamma) - J(\alpha, \alpha^{-1}) \delta(\alpha \beta) \\
& + 2\delta(\alpha) \delta(\beta) + \delta(\alpha) \delta(\gamma) + \delta(\beta) \delta(\gamma).
\end{align*}
Here we have made use of several tautological identities like $J(\alpha \beta, \gamma) \delta(\beta) = J(\alpha, \gamma) \delta(\beta)$. A similar procedure yields
\begin{align*}
B^* =& J(\alpha, \beta \gamma) J(\beta, \gamma) - J(\beta, \gamma) \delta(\alpha) - J(\alpha, \gamma) \delta(\beta) -J(\alpha, \beta) \delta(\gamma) - J(\gamma, \gamma^{-1}) \delta(\beta \gamma) \\
& + \delta(\alpha) \delta(\beta) + \delta(\alpha) \delta(\gamma) + 2\delta(\beta) \delta(\gamma).
\end{align*}
We thus see that $A^*-A=B^*-B$, and so $A=B$ is equivalent to $A^*=B^*$.
\end{proof}

\section{Properties of Jacobi sums} \label{s:ident}

Fix a finite field $\FF$. We now verify that the Jacobi sum $\JJ$ actually satisfies the conditions (A), (B), and (C). The first follows immediately from the definition.

\begin{proposition} \label{prop:jacobi2}
$\JJ$ satisfies (B).
\end{proposition}

\begin{proof}
Let $\alpha, \beta, \gamma \in \MM$ be given, and consider the sum
\begin{displaymath}
S = \frac{1}{(q-1)^2} \sum_{\substack{x+y+z=1,\\ x,y,z \in \FF^*}} \alpha(x) \beta(y) \gamma(z).
\end{displaymath}
We can rewrite this as
\begin{displaymath}
S = \frac{1}{q-1} \sum_{\substack{w+z=1,\\ w \in \FF, z \in \FF^*}} S(w) \gamma(z), \qquad
S(w) = \frac{1}{q-1} \sum_{\substack{x+y=w, \\ x,y \in \FF^*}} \alpha(x) \beta(y)
\end{displaymath}
If $w \ne 0$ then a change of variables gives $S(w)=(\alpha \beta)(w) \JJ(\alpha, \beta)$. We also have $S(0)=\alpha(-1) \delta(\alpha \beta)$. We therefore find
\begin{align*}
S
&= \JJ(\alpha, \beta) \JJ(\alpha \beta, \gamma) + \frac{\alpha(-1) \delta(\alpha \beta)}{q-1} \\
&= \JJ(\alpha, \beta) \JJ(\alpha \beta, \gamma) - \JJ(\alpha, \alpha^{-1}) \delta(\alpha \beta) + \delta(\alpha) \delta(\beta).
\end{align*}
In the second step, we used the computation of $\JJ(\alpha, \alpha^{-1})$ \cite[\S 8.3]{IR}. On the other hand, we can carry out a similar computation where the outer sum is over $x+w=1$ and the inner sum $S(w)$ is over $y+z=w$. This gives the same result, except with $\alpha$ and $\gamma$ interchanged. We thus see that (B) holds by Proposition~\ref{prop:reform}.
\end{proof}

\begin{proposition} \label{prop:jacobi3}
$\JJ$ satisfies (C).
\end{proposition}

\begin{proof}
Let $\alpha_1, \alpha_2, \alpha_3, \alpha_4 \in \MM$ be given, and consider the sum
\begin{displaymath}
S = \sum_{\beta \in \MM} \JJ(\alpha_1 \beta, \alpha_2 \beta^{-1}) \JJ(\alpha_3 \beta, \alpha_4 \beta^{-1})
\end{displaymath}
We have
\begin{align*}
S &= (q-1)^{-2} \sum_{\beta} \sum_{x,y \ne 0,1} (\alpha_1 \beta)(x) (\alpha_2 \beta^{-1})(1-x) (\alpha_3 \beta)(y) (\alpha_4 \beta^{-1})(1-y) \\
&= (q-1)^{-2} \sum_{x,y \ne 0,1} \bigg( \alpha_1(x) \alpha_2(1-x) \alpha_3(y) \alpha_4(1-y) \sum_{\beta} \beta \big( \frac{xy}{(1-x)(1-y)} \big) \bigg).
\end{align*}
Now, the inner sum vanishes unless the argument to $\beta$ is equal to~1, in which case it is equal to $q-1$. The equation $xy=(1-x)(1-y)$ is equivalent to $x+y=1$. We thus obtain
\begin{displaymath}
S = (q-1)^{-1} \sum_{\substack{x+y=1, \\ x,y \ne 0,1}} \alpha_1(x) \alpha_2(y) \alpha_3(y) \alpha_4(x) = \JJ(\alpha_1 \alpha_4, \alpha_2 \alpha_3),
\end{displaymath}
and so the result follows.
\end{proof}

\section{Initial analysis} \label{s:initial}

Fix, for the duration of \S \ref{s:initial}, a finite abelian group $M$ and a Jacobi function $J$ on $M$. We let $m$ be the order of $M$, and we let $\hat{M}$ be the Pontryagin dual of $M$. Our goal is to prove the following result:

\begin{proposition} \label{prop:initial}
There is a subset $S$ of $\hat{M}$ and a function $i \colon S \to \hat{M}$ such that
\begin{displaymath}
J(\alpha, \beta) = \frac{1}{m} \sum_{x \in S} \alpha(i(x)) \beta(i(x)x^{-1}).
\end{displaymath}
The set $S$ is closed under inversion (i.e., $x \in S$ implies $x^{-1} \in S$), and the function $i$ satisfies the identity $i(x)=xi(x^{-1})$. Moreover, we are in one of the following two cases:
\begin{enumerate}
\item There is an element $c \in \hat{M}$ satisfying $c^2=1$ such that $S=\hat{M} \setminus \{c\}$, and $i \colon \hat{M} \setminus \{c\} \to \hat{M} \setminus \{1\}$ is a bijection.
\item We have $S=\hat{M}$ and $i \colon \hat{M} \to \hat{M}$ is a bijection.
\end{enumerate}
\end{proposition}

\begin{remark} \label{rmk:classical}
In the classical setting, where $J$ is the Jacobi sum $\JJ$ on a finite field $\FF$, we are in case (a) with $c=-1$ and $i(x)=x/(1+x)$. For the $a=1$ Jacobi function in \S \ref{ss:rmk}(a), we are in case~(b).
\end{remark}

The proof is contained in a series of lemmas below. We first make some definitions. The \defn{convolution} of two functions $f,g \colon M \to \bC$ is the function $f \ast g \colon M \to \bC$ given by
\begin{displaymath}
(f \ast g)(\alpha) = \sum_{\alpha=\beta \gamma} f(\beta) g(\gamma).
\end{displaymath}
The \defn{Fourier transform} of $f$ is the function $\hat{f} \colon \hat{M} \to \bC$ given by
\begin{displaymath}
\hat{f}(x) = \sum_{\alpha \in M} f(\alpha) \alpha(x).
\end{displaymath}
We have the usual identity $(f \ast g)^{\wedge} = \hat{f} \cdot \hat{g}$. For $\alpha \in M$, define $Q_{\alpha} \colon M \to \bC$ by
\begin{displaymath}
Q_{\alpha}(\beta) = J(\alpha \beta^{-1}, \beta).
\end{displaymath}
These functions have the following important property:

\begin{lemma} \label{lem:init-1}
We have $Q_{\alpha} \ast Q_{\beta}=Q_{\alpha \beta}$ for all $\alpha, \beta \in M$.
\end{lemma}

\begin{proof}
We have
\begin{align*}
(Q_{\alpha} \ast Q_{\beta})(\chi)
&= \sum_{\gamma \in M} Q_{\alpha}(\gamma) Q_{\beta}(\chi \gamma^{-1}) \\
&= \sum_{\gamma \in M} J(\alpha \gamma^{-1}, \gamma) J(\beta \chi^{-1} \gamma, \chi \gamma^{-1}) \\
&= J(\alpha \beta \chi^{-1}, \chi) = Q_{\alpha \beta}(\chi),
\end{align*}
as required. In the third step we used condition (C).
\end{proof}

\begin{lemma} \label{lem:init-2}
There is a subset $S \subset \hat{M}$ and a function $i \colon S \to \hat{M}$ such that
\begin{displaymath}
J(\alpha, \beta) = \frac{1}{m} \sum_{x \in S} \alpha(i(x)) \beta(i(x) x^{-1}).
\end{displaymath}
\end{lemma}

\begin{proof}
From Lemma~\ref{lem:init-1}, we find
\begin{displaymath}
\hat{Q}_{\alpha} \cdot \hat{Q}_{\beta} = \hat{Q}_{\alpha \beta}.
\end{displaymath}
In particular, we see that $\hat{Q}_1$ takes values in the set $\{0,1\}$. Let $S$ be the set of elements $x \in \hat{M}$ where $\hat{Q}_1(x)=1$. We have $\hat{Q}_{\alpha}^m =\hat{Q}_1$, and so $\hat{Q}_{\alpha}$ vanishes off of $S$. For $x \in S$, the function $\alpha \mapsto \hat{Q}_{\alpha}(x)$ is a character of $M$, and thus given by pairing with some element $i(x) \in \hat{M}$. By Fourier inversion, we have
\begin{displaymath}
Q_{\alpha}(\beta)\
= \frac{1}{m} \sum_{x \in \hat{M}} \hat{Q}_{\alpha}(x) \beta^{-1}(x)
= \frac{1}{m} \sum_{x \in S} \alpha(i(x)) \beta^{-1}(x).
\end{displaymath}
Since $J(\alpha, \beta) = Q_{\alpha \beta}(\beta)$, the result follows.
\end{proof}

\begin{lemma} \label{lem:init-3}
The set $S$ is closed under inversion, and $i(x)=x i(x^{-1})$ for all $x \in S$.
\end{lemma}

\begin{proof}
Let $X \subset \hat{M}^2$ be the set of all pairs $(i(x), i(x)x^{-1})$ with $x \in S$, and let $Y$ be defined similarly but with the coordinates switched. Since $x \mapsto (i(x), i(x)x^{-1})$ is injective, we have
\begin{displaymath}
J(\alpha, \beta) = \frac{1}{m} \sum_{(x,y) \in X} \alpha(x) \beta(y).
\end{displaymath}
By condition (A), the same holds using $Y$ in place of $X$. Thus the indicator functions of $X$ and $Y$ have the same Fourier transform, and so $X=Y$. Therefore, given $x \in S$, there exists $y \in S$ such that $i(x)=i(y)y^{-1}$ and $i(x)x^{-1}=i(y)$. These equations imply $y=x^{-1}$, and so the result follows.
\end{proof}

From the above lemma, we see that $x=i(x) i(x^{-1})^{-1}$, which shows that $i$ is injective. Let $T \subset \hat{M}$ be the image of $i$, so that $i \colon S \to T$ is a bijection. We have
\begin{displaymath}
J(1, \alpha)=\frac{1}{m} \sum_{x \in T} \alpha(x).
\end{displaymath}

\begin{lemma} \label{lem:init-4}
We have $J^*(1, \alpha) \ne 0$ for all non-trivial $\alpha$.
\end{lemma}

\begin{proof}
Suppose there is $\alpha \ne 1$ such that $J^*(1, \alpha)=0$. This implies $J(1, \alpha)=1$, and so
\begin{displaymath}
\sum_{x \in T} \alpha(x)=m.
\end{displaymath}
Since each $\alpha(x)$ is a root of unity and $\# T \le m$, this equation can only hold if $T=\hat{M}$ and $\alpha=1$, a contradiction.
\end{proof}

\begin{lemma} \label{lem:init-5}
We are in either case (a) or (b) from Proposition~\ref{prop:initial}.
\end{lemma}

\begin{proof}
For any $\alpha \in M$, condition (B) gives
\begin{displaymath}
J^*(1, \alpha) \cdot J^*(1, \alpha) = J^*(1, 1) \cdot J^*(1, \alpha).
\end{displaymath}
We thus find that $J^*(1, \alpha)=J^*(1,1)$ for all $\alpha$; indeed, if $\alpha$ is non-trivial then we can cancel one $J^*(1, \alpha)$ factor from each side by Lemma~\ref{lem:init-4}. Putting $\kappa=J(1,1)-1$, we find $J(1, \alpha)=\kappa+\delta(\alpha)$ for all $\alpha \in M$. Thus
\begin{displaymath}
\kappa+\delta(\alpha) = \frac{1}{m} \sum_{x \in T} \alpha(x)
\end{displaymath}
holds for all $\alpha$. Taking the Fourier transform, we see that $T$ contains all non-identity elements of $\hat{M}$. If $T$ contains~1 then $T=\hat{M}$ and we are in case (b) of the proposition. If $T$ does not contain~1 then $\# S=\# T = m-1$, and so $S=\hat{M} \setminus \{c\}$ for some $c$. Since $S$ is closed under inversion, it follows that $c^2=1$, and so we are in case (a).
\end{proof}

\section{The first case} \label{s:a}

Let $J$ be a Jacobi function on $M$, and suppose we are in the first case of Propopsition~\ref{prop:initial}; that is, there is an element $c \in \hat{M}$ satisfying $c^2=1$ and a bijection $i \colon \hat{M} \setminus \{c\} \to \hat{M} \setminus \{1\}$ as in the proposition. Put $F=\hat{M} \sqcup \{0\}$. We extend the multiplication law on $\hat{M}$ to all of $F$ by defining $0 \cdot x = x \cdot 0 = x$ for all $x \in F$. We define a binary operation $\oplus$ on $F$ by taking 0 to be the identity, and declaring
\begin{displaymath}
x \oplus y = \begin{cases}
0 & \text{if $x=cy$} \\
xi(x/y)^{-1} & \text{otherwise} \end{cases}
\end{displaymath}
when $x$ and $y$ are both non-zero. Our goal is to prove the following proposition, which verifies the main theorem in this case:

\begin{proposition}
With the above structure, $F$ is a field. Moreover, we have
\begin{displaymath}
J(\alpha, \beta) = \frac{1}{m} \sum_{\substack{x \oplus y=1,\\ x,y \in F \setminus \{0\}}} \alpha(x) \beta(y)
\end{displaymath}
for all $\alpha, \beta \in M$.
\end{proposition}

\begin{remark}
Suppose we are in the classical case, where $J=\JJ$ is the Jacobi sum on a finite field $\FF$. From Remark~\ref{rmk:classical}, we know that $c=-1$ and $i(x)=x/(x+1)$. It follows that our $\oplus$ coincides with the usual addition law on $\FF$.
\end{remark}

The proposition will be proved in a series of lemmas.

\begin{lemma} \label{lem:a1}
We have the following:
\begin{enumerate}
\item The operation $\oplus$ is commutative.
\item The distributive law holds.
\item For every $x \in F$, the element $cx \in F$ is its inverse under $\oplus$.
\end{enumerate}
\end{lemma}

\begin{proof}
(a) It is clear that $x \oplus y = y \oplus x$ if either $x$ or $y$ is equal to~0. If $x/y=c$ then $y/x=c$ as well, since $c^2=1$, and so $x \oplus y = y \oplus x = 0$. Finally, suppose that $x$ and $y$ are not zero, and $x/y \ne c$. Then
\begin{displaymath}
x \oplus y = xi(x/y)^{-1} = x( (x/y) i(y/x))^{-1} = y \oplus x,
\end{displaymath}
where in the second step we used the identity $i(z)=zi(z^{-1})$.

(b) We must show
\begin{displaymath}
z (x \oplus y) = (zx) \oplus (zy)
\end{displaymath}
for all $x,y,z \in F$. If any of $x$, $y$, or $z$ vanishes, then the result is clear; thus suppose this is not the case. If $x/y=c$ then $(zx)/(zy)=c$, and both sides are~0. Suppose now that $x/y \ne c$, which implies $(zx)/(zy) \ne c$. Then
\begin{displaymath}
z(x \oplus y) = zx i(x/y)^{-1}=(zx)i((zx)/(zy))^{-1} = (zx) \oplus (zy),
\end{displaymath}
and so the result follows.

(c) follows directly from the definition.
\end{proof}

\begin{lemma} \label{lem:a2}
We have
\begin{displaymath}
J(\alpha, \beta) = \frac{1}{m} \sum_{\substack{x \oplus y = 1,\\ x,y \in F \setminus \{0\}}} \alpha(x) \beta(y).
\end{displaymath}
\end{lemma}

\begin{proof}
Let
\begin{align*}
X &= \{ (x,y) \in \hat{M}^2 \mid x \oplus y=1\}, \\
Y &= \{ (i(z), i(z)z^{-1}) \mid z \in \hat{M} \setminus \{c\} \}.
\end{align*}
We claim $X=Y$. Indeed, if $(x,y) \in X$ then $x/y \ne c$ and $xi(x/y)^{-1}=1$, and so $x=i(x/y)$. Thus, putting $z=x/y$, we have $x=i(z)$ and $y=i(z)z^{-1}$, and so $(x,y) \in Y$. On the other hand, if $z \in \hat{M} \setminus \{c\}$ and we put $x=i(z)$ and $y=i(z)z^{-1}$ then $xi(x/y)^{-1}=1$, and so $(x,y) \in X$. This establishes the claim. We thus have
\begin{displaymath}
\frac{1}{m} \sum_{\substack{x \oplus y = 1,\\ x,y \in F \setminus \{0\}}} \alpha(x) \beta(y) = \frac{1}{m} \sum_{z \in \hat{M} \setminus \{c\}} \alpha(i(z)) \beta(i(z)z^{-1}) = J(\alpha, \beta),
\end{displaymath}
and so the result follows.
\end{proof}

\begin{lemma} \label{lem:a3}
For any $\alpha, \beta, \gamma \in M$, we have
\begin{displaymath}
\sum_{\substack{(x \oplus y) \oplus z=1, \\ x,y,z \in F \setminus \{0\}}} \alpha(x) \beta(y) \gamma(z) =
\sum_{\substack{x \oplus (y \oplus z)=1, \\ x,y,z \in F \setminus \{0\}}} \alpha(x) \beta(y) \gamma(z)
\end{displaymath}
\end{lemma}

\begin{proof}
We argue as in the proof of Proposition~\ref{prop:jacobi2}, except in reverse. Let $A$ be the left side of the equation divided by $m^2$, and similarly define $B$ for the right side. We have
\begin{displaymath}
A = \frac{1}{m} \sum_{\substack{w \oplus z = 1,\\ w \in F, z \in F \setminus \{0\}}} S(w) \gamma(z), \qquad
S(w) = \frac{1}{m} \sum_{\substack{x \oplus y=w, \\ x,y \in F \setminus \{0\}}} \alpha(x) \beta(y).
\end{displaymath}
If $w \ne 0$ then $S(w)=(\alpha \beta)(w) J(\alpha, \beta)$, just as in the proof of Proposition~\ref{prop:jacobi2}; the key point here is that the distributive law holds. From the definition of $\oplus$, it is clear that $x \oplus y=0$ if and only if $x=cy$. Thus
\begin{displaymath}
S(0) = \frac{1}{m} \sum_{y \in \hat{M}} \alpha(cy) \beta(y)
= \frac{\alpha(c) \delta(\alpha \beta)}{m}.
\end{displaymath}
We note that
\begin{displaymath}
J(\alpha, \alpha^{-1}) = \frac{1}{m} \sum_{x \in \hat{M} \setminus \{c\}} \alpha(x) = \delta(\alpha) - \frac{\alpha(c)}{m},
\end{displaymath}
and so
\begin{displaymath}
S(0) = \delta(\alpha) \delta(\beta) - J(\alpha, \alpha^{-1}) \delta(\alpha \beta).
\end{displaymath}
Putting this all together, we find
\begin{displaymath}
A = J(\alpha, \beta) J(\alpha \beta, \gamma) - J(\beta, \beta^{-1}) \delta(\alpha \beta) + \delta(\alpha) \delta(\beta).
\end{displaymath}
A similar computation gives
\begin{displaymath}
B = J(\beta, \gamma) J(\alpha, \beta \gamma) - J(\beta, \beta^{-1}) \delta(\beta \gamma) + \delta(\beta) \delta(\gamma).
\end{displaymath}
The equality $A=B$ is now a consequence of condition (B) and Proposition~\ref{prop:reform}.
\end{proof}

\begin{lemma} \label{lem:a4}
The operation $\oplus$ is associative.
\end{lemma}

\begin{proof}
Put
\begin{align*}
X &= \{ (x,y,z) \in (F \setminus \{0\})^3 \mid (x \oplus y) \oplus z = 1\} \\
Y &= \{ (x,y,z) \in (F \setminus \{0\})^3 \mid x \oplus (y \oplus z) = 1\}.
\end{align*}
The previous lemma shows that the indicator functions of $X$ and $Y$ have equal Fourier transform. By Fourier inversion, it follows that $X=Y$. We thus find
\begin{displaymath}
(x \oplus y) \oplus z = 1 \iff x \oplus (y \oplus z) = 1
\end{displaymath}
if $x,y,z \in F$ are non-zero. Of course, if any of $x$, $y$, or $z$ vanishes then this obviously holds. Since the distributive law holds, for $a \in F \setminus \{0\}$ we find
\begin{displaymath}
(x \oplus y) \oplus z = a \iff x \oplus (y \oplus z) = a.
\end{displaymath}
But this implies that the same statement holds for $a=0$, since this is the only remaining value; indeed, if $(x \oplus y) \oplus z=0$ then the above shows that $x \oplus (y \oplus z)$ cannot be any element of $F \setminus \{0\}$, and so it must be zero. This completes the proof.
\end{proof}

\section{The second case} \label{s:b}

Let $J$ be a Jacobi function on $M$, and suppose we are in the second case of Propopsition~\ref{prop:initial}. Thus there is a bijection $i \colon \hat{M} \to \hat{M}$ as in the proposition. Our goal is to prove the following:

\begin{proposition}
The group $M$ is trivial.
\end{proposition}

This will complete the proof of the main theorem. To prove the proposition, we follow the same general plan as in \S \ref{s:a}, at least initially. We define a binary operation $\oplus$ on $\hat{M}$ by
\begin{displaymath}
x \oplus y = xi(x/y)^{-1}.
\end{displaymath}
Unlike \S \ref{s:a}, there are no edge cases in this definition.

\begin{lemma}
We have the following:
\begin{enumerate}
\item The operation $\oplus$ is commutative.
\item The distributive law holds.
\item The operation $\oplus$ is cancellative, i.e., $x \oplus y = x \oplus z$ implies $y=z$.
\end{enumerate}
\end{lemma}

\begin{proof}
The first two statements are proved as in Lemma~\ref{lem:a1}. Now suppose we have $x$, $y$, and $z$ as in (c). Then $xi(x/y)^{-1}=xi(x/z)^{-1}$. Since $M$ is a group and $i$ is bijective, it follows that $y=z$.
\end{proof}

\begin{lemma}
We have
\begin{displaymath}
J(\alpha, \beta) = \frac{1}{m} \sum_{\substack{x \oplus y = 1,\\ x,y \in \hat{M}}} \alpha(x) \beta(y).
\end{displaymath}
\end{lemma}

\begin{proof}
This is proved like Lemma~\ref{lem:a2}.
\end{proof}

\begin{lemma}
For any $\alpha, \beta, \gamma \in M$, we have
\begin{displaymath}
\sum_{\substack{(x \oplus y) \oplus z=1, \\ x,y,z \in \hat{M}}} \alpha(x) \beta(y) \gamma(z) =
\sum_{\substack{x \oplus (y \oplus z)=1, \\ x,y,z \in \hat{M}}} \alpha(x) \beta(y) \gamma(z)
\end{displaymath}
\end{lemma}

\begin{proof}
The proof is similar to that of Lemma~\ref{lem:a3}, but there are a few differences worth noting. As before, let $A$ and $B$ be the two sides of the equation, divided by $m^2$. We have
\begin{displaymath}
A = \frac{1}{m} \sum_{\substack{w \oplus z = 1,\\ w,z \in \hat{M}}} S(w) \gamma(z), \qquad
S(w) = \frac{1}{m} \sum_{\substack{x \oplus y=w, \\ x,y \in \hat{M}}} \alpha(x) \beta(y).
\end{displaymath}
The main difference to the previous proof is that there is no $w=0$ case here. Thus $S(w)=(\alpha \beta)(w) J(\alpha, \beta)$ for all $w$, and we find
\begin{displaymath}
A = J(\alpha, \beta) J(\alpha \beta, \gamma).
\end{displaymath}
Similarly,
\begin{displaymath}
B = J(\alpha, \beta \gamma) J(\beta, \gamma).
\end{displaymath}
Since
\begin{displaymath}
J(\alpha, \alpha^{-1})=\frac{1}{m} \sum_{x \in \hat{M}} \alpha(x) = \delta(\alpha)
\end{displaymath}
holds for all $\alpha \in M$, Proposition~\ref{prop:reform} shows $A=B$.
\end{proof}

\begin{lemma}
The operation $\oplus$ is associative.
\end{lemma}

\begin{proof}
The proof of Lemma~\ref{lem:a4} applies.
\end{proof}

\begin{lemma}
There is an identity element 0 for $\oplus$.
\end{lemma}

\begin{proof}
Let $x$ be any element of $\hat{M}$. The function $t_x \colon \hat{M} \to \hat{M}$ defined by $y \mapsto x \oplus y$ is injective (since $\oplus$ is cancellative), and therefore a bijection (since $\hat{M}$ is finite). Since $t_x$ is a permutation of the $m$-element set $\hat{M}$, it follows that $t_x^{m!}$ is the identity. But $t_x^n=t_{nx}$ for any $n \ge 1$, and so we see that $m! \cdot x$ is the additive identity.
\end{proof}

\begin{lemma}
The set $\hat{M}$ has one element.
\end{lemma}

\begin{proof}
The equation $x=x \oplus 0$ exactly means $x=xi(x/0)^{-1}$; note that 0 is an element of the group $\hat{M}$, and so $0^{-1}$ is well-defined. We thus see that $i(x/0)=1$, and so $x=0 \cdot i^{-1}(1)$. Since this holds for all $x$, the result follows.
\end{proof}

\end{document}